\documentclass[letterpaper]{amsart}

\usepackage[english]{babel}
\usepackage{amsmath}
\usepackage{graphicx}
\usepackage[colorinlistoftodos]{todonotes}
\usepackage{amssymb}
\usepackage{amsthm} 	
\usepackage{thmtools} 	
\usepackage{amscd} 		
\usepackage{dsfont}
\usepackage{fixltx2e}
\usepackage{hyperref} 
\usepackage{cite}
\usepackage{tikz-cd}
\usepackage{listings}

\newtheorem{lemma}{Lemma}
\newtheorem{thm}{Theorem}
\newtheorem*{thm*}{Theorem}

\newtheorem*{conj*}{Conjecture}

\newtheorem{prop}[thm]{Proposition}
\theoremstyle{remark} \newtheorem{remark}{Remark}
\theoremstyle{definition} \newtheorem{ex}{Example} 
\newtheorem*{defn}{Definition}
\DeclareMathOperator{\Ima}{Im}

\newcommand{\bbA}{\mathbb{A}}
\newcommand{\bbC}{\mathbb{C}}

\newcommand{\bbP}{\mathbb{P}}

\newcommand{\calO}{\mathcal{O}}
\newcommand{\calT}{\mathcal{T}}
\newcommand{\calN}{\mathcal{N}}

\newcommand{\calH}{\mathcal{H}}

\newcommand{\calIC}{\mathcal{IC}_X^\bullet}
\newcommand{\calICT}{\mathcal{IC}_{\tau_{m,n,k}}^\bullet}
\newcommand{\calCC}{\mathcal{CC}}
\newcommand{\calF}{\mathcal{F}}

\newcommand{\codim}{\operatorname{codim}}

\newcommand{\tr}{\operatorname{trace}}

\newcommand{\ind}{\mathds{1}}
\newcommand\numberthis{\addtocounter{equation}{1}\tag{\theequation}}

\title[Euler Obstruction, Chern Class and IC Characteristic Cycles]{Local Euler Obstructions, Chern-Mather classes and IC Characteristic Cycles of Determinantal Varieties}

 \author{Xiping Zhang}

\date{\today}

\begin{document}

\bibliographystyle{plain}

\begin{abstract}\ 
For $m\geq n$, Let $K$ be an algebraically closed field, and define $\tau_{m,n,k}$ to be the set of $m\times n$ matrices over $K$ with kernel dimension $\geq k$. This is a projective subvariety 
of $\bbP^{mn-1}$, and is usually called determinantal variety. In most cases $\tau_{m,n,k}$ is singular with singular locus 
$\tau_{m,n,k+1}$. 
In this paper we compute the local Euler obstruction of $\tau_{m,n,k}$,
and we prove that the characteristic cycle of the intersection cohomology
complex of $\tau_{m,n,k}$ is irreducible.
We also give an explicit formula for the Chern-Mather class of $\tau_{m,n,k}$
as a class in the projective space. The irreducibility of the intersection cohomology characteristic cycle
follows from the explicit computation of the local Euler obstruction, a
study of the `Tjurina transforms' of determinantal varieties, and the Kashiwara-Dubson's 
microlocal index theorem. 

Our explicit formulas are based on calculations of
degrees of certain Chern classes of the universal bundles over the Grassmannian.
We use Macaulay2 to exhibit examples of the Chern-Mather class and the 
class of the projectived characteristic cycle of $\tau_{m,n,k}$ for some small values of $m,n,k$.

Over the complex numbers, the local Euler obstruction of $\tau_{m,n,k}$ 
was recently computed by N.~Grulha, T.~Gaffney and M.~Ruas by methods
in complex geometry~\cite{NG-TG}.
\end{abstract}

\maketitle
\section{Introduction}
Let $K$ be an algebraic closed field.
For $m\geq n$, we define $\tau_{m,n,k}$ to be the set of all $m$ by $n$ matrices over 
$K$ which have kernel
of dimension $\geq k$. This is an irreducible projective subvariety of $\bbP^{mn-1}$, and in most cases $\tau_{m,n,k}$ is singular with singular locus $\tau_{m,n,k+1}$. The varieties $\tau_{m,n,k}$ 
are called determinantal varieties, and have been the object of intense study. 
(See e.g.,~\cite{MR1143555},~\cite[\S14.4]{INT},~\cite[Lecture 9]{AG-JH}.)

One very important local invariant on a singular variety is the local Euler obstruction
$Eu_X$. It was first defined by R.D. MacPherson for compact complex algebraic varieties 
in \cite{MAC}, as an ingredient in his construction of a natural transformation 
from the functor of constructible functions to homology. This is the unique natural transform that normalize to the total Chern class on smooth varieties, and MacPherson denoted it by $c_*$. 

In the same paper, MacPherson also defined the Chern-Mather class 
$c_M(X)$ for any (possibly singular)
variety $X$ embedded in an ambient space $M$. The natural transformation $c_*$ is defined by sending $Eu_X$ to $c_M(X)$.

MacPherson worked over $\bbC$, and the definition of local Euler obstruction given 
in \cite{MAC} relies on complex geometry.
Both the local Euler obstruction and the Chern-Mather class can be defined over 
arbitrary algebraically closed fields $K$ (cf. \cite{MR629121}), while the natural transformation $c_*$ extends to arbitrary base fields of characteristic $0$ (cf.~\cite{Kennedy}). Also one can translate the whole story from the homology group to the Chow group (cf.~\cite[19.1.7]{INT}). 

For $K=\bbC$ is the complex number field, N.~Grulha, T.~Gaffney, and M.~Ruas computed in \cite{NG-TG} the local 
Euler obstruction of $\tau_{m,n,k}$, using methods from complex geometry: they found (\cite[Theorem~1.16]{NG-TG}) that the values of the local Euler obstruction of
$\tau_{m,n,k}$ along strata determined by smaller determinantal varieties are given by
binomial coefficients, so that they satisfy a basic recursion formula.
(Actually, Grulha, Gaffney and Ruas~worked with the affine determinantal variety 
$\Sigma_{m,n,k}$, that is, the affine cone over $\tau_{m,n,k}$; it is easy to see that this
does not affect the local Euler obstruction.)

In this paper we use a direct intersection-theoretic approach and prove the following formulas for the local Euler obstruction of $\tau_{m,n,k}$, valid over
any algebraically closed base field $K$.

\begin{thm*}[Theorem~\ref{thm; eulerobstruction}]
For any $\varphi\in\tau_{m,n,k+i}\smallsetminus\tau_{m,n,k+i+1}$, the local Euler obstruction of $\tau_{m,n,k}$ at $\varphi$ is:
\begin{align*}
Eu_{\tau_{m,n,k}}(\varphi)
&=\int_{G(k,k+i)\times G(i,m-n+k+i)} c^{-1}(S_1^\vee\otimes Q_2)c^{-1}(Q_1^\vee\otimes S_2) \\
&=\int_{G(k,k+i)\times G(i,m-n+k+i)} c_{top}(S_1^\vee\otimes S_2)c_{top}(Q_1^\vee\otimes Q_2) 
\end{align*}
\end{thm*}
Here, $S_l$ and $Q_l$, $l=1,2$ denote respectively the universal subbundles and
quotient bundles over the two factors in the product $G(k,k+i)\times G(i,m-n+k+i)$.

A direct observation from the above theorem is that the Euler obstruction of $\tau_{m,n,k}$ can be obtained from smaller determinantal varieties. For any $\varphi\in \tau_{m,n,j}\smallsetminus \tau_{m,n,j+1}$, and any $\varphi'\in \tau_{m-n+j+1, j+1,j}$, we have;
\[
Eu_{\tau_{m,n,k}}(\varphi)=Eu_{\tau_{m-n+j+1,j+1,k}}(\varphi') \/.
\]

From this observation  we are able to derive the basic recursion formula (Pascal's Triangle)
satisfied by the local Euler obstruction,
thereby extending the result of \cite[Theorem~1.16]{NG-TG} to arbitrary
algebraically closed fields. We define $Eu_{m,n,k}:=Eu_{\tau_{m,n,k}}(\varphi)$ for any $\varphi\in \tau_{m,n,n-1}$ to be the local Euler obstruction of $\tau_{m,n,k}$ along $\tau_{m,n,n-1}$. 
\begin{thm*}[Theorem~\ref{thm; pascal}] 
We have the following recursive formula for $k=0,1,\cdots n-2$:
\[
Eu_{m,n,k}+Eu_{m,n,k+1}=Eu_{m+1,n+1,k+1} \/.
\]
Hence we have 
\[
Eu_{m,n,k}=\binom{n-1}{k}
\]
and therefore
\begin{equation}
\label{eq; Eubinom}
Eu_{\tau_{m,n,k}}(\varphi)=\binom{k+i}{i}
\end{equation}
for any $\varphi\in \tau_{m,n,k+i}\smallsetminus \tau_{m,n,k+i+1}$.
\end{thm*}

As an application of this result (over~$\bbC$), we prove the following result:
\begin{thm*}[Theorem~\ref{thm; charcycle}] 
The characteristic cycle associated with the
intersection cohomology sheaf of $\tau_{m,n,k}$ equals the conormal cycle
of $\tau_{m,n,k}$ (hence it is irreducible).
\end{thm*}
We remind the reader that for every irreducible subvariety $X\subset M$ of a
smooth complex algebraic variety,
Goresky and MacPherson defined a sheaf of bounded complexes on $M$, denoted by 
$\calIC$, 
called the \textit{Intersection cohomology Sheaf} of $X$. 
(For more details about the intersection cohomology sheaf and intersection homology,
see~\cite{G-M1} and~\cite{G-M2}.)
As  $\calIC$ is a constructible sheaf with respect to any Whitney 
stratification $\sqcup_{i\in I} S_i$ of $M$, one can assign to it a cycle in the cotangent 
bundle $T^*M$. 
We will call this cycle the `$IC$ characteristic cycle', denoted $\calCC(\calIC)$. The $IC$ characteristic cycle of $X$ can be expressed as a linear combination of the conormal cycles of the strata:
\[
\calCC(\calIC)=\sum_{i\in I} c_i(\calIC)[T^*_{\overline{S_i}}M]\/.
\]
Here the integer coefficients  $c_i(\calIC)$ are called the \textit{Microlocal Multiplicities}. 
(See~\cite[Section 4.1]{Dimca} for an explicit construction of the
$IC$ characteristic cycle and the microlocal multiplicities.) 

In our context, for $X=\tau_{m,n,k}\subset M=\bbP^{mn-1}$,
\[
\calCC(\calICT)=\sum_{i\ge k} c_i(\calICT)[T^*_{\tau_{m,n,i}}\bbP^{mn-1}] \/.
\]
As a corollary of formula~\eqref{eq; Eubinom}, we prove that $c_i(\calICT)=\delta_{ik}$,
establishing that $\calCC(\calICT)$ is irreducible.
As pointed out in \cite[Rmk 3.2.2]{Jones}, this is a rather unusual phenomenon. It is known to be true for Schubert varieties in a Grassmannian, for certain Schubert varieties in flag manifolds of types B, C, and D, and for theta divisors of Jacobians (cf. \cite{MR1642745}). 
By the above result, determinantal varieties also share this 
property.

Our main tool is the deep microlocal index formula of Kashiwara and Dubson
(\cite[Theorem 6.3.1]{Ka}, \cite[Theorem 3]{Du}) and an explicit study of the
so-called `Tjurina transform' $\nu\colon \hat\tau_{m,n,k}\to \tau_{m,n,k}$ of the 
variety $\tau_{m,n,k}$.
Our study of the Tjurina transform also implies that (over any algebraically closed field)
\[
c_M(\tau_{m,n,k})=\nu_*(\hat\tau_{m,n,k}).
\]
This formula is more efficient than a direct application of the definition,
which uses the Nash blow-up of $\tau_{m,n,k}$: this would involve computations
in a product of Grassmannians, while the Tjurina transform only involves computations in a single Grassmannian.

In Theorem \ref{thm; cmalgorithm} we give an explicit formula for $c_M(\tau_{m,n,k})$, and 
in \S\ref{example} we use Macaulay2 to exhibit some examples.
In \S\ref{example} we also give explicit examples of computations of the projectived conormal cycles
(i.e., by the above result, of the projectived $IC$ characteristic cycles) of determinantal varieties.

We note that the class $\nu_*(\hat\tau_{m,n,k})$ is one ingredient in our computation
of the Chern-Schwartz-MacPherson classes of determinantal varieties (cf. \cite{XP1}),
although this class was not identified as the Chern-Mather class in our previous work.

I would like to thank Paolo Aluffi for all the help and support. I would also like to thank Terence Gaffney and Nivaldo G. Grulha Jr.~for the helpful discussions during my visit to Northeastern University.  


\section{Local Euler Obstruction and Chern-Mather Class}
\label{S; setup}
\subsection{MacPherson's Natural Transformation $c_*$}
\label{S; c_*}
In \cite{MAC} MacPherson proved the existence and uniqueness of Chern classes
for possibly singular complex algebraic varieties, which was conjectured earlier
by Deligne and Grothendieck. Let $F$ be the functor of constructible functions, and $A$ be the functor of Chow groups. They are both functors from the category 
\{complex algebraic varieties, proper morphisms\} to the category of abelian groups. 

\begin{thm}[R. D. MacPherson, 1973~\cite{MAC}]
There is a unique natural transformation $c_*$ from the functor
$F$ to the Chow group functor $A$ such that if $X$ is smooth, then 
$c_*(\ind_X)=c(\mathcal{T}_X)\cap [X]$, where $\mathcal{T}_X$ 
is the tangent bundle of $X$.
\end{thm}
\begin{remark}
MacPherson's original work was on homology groups instead of Chow groups, but one can change settings and get a Chow group version of the theorem. Cf.~\cite[19.1.7]{INT}. 
\end{remark}

In the proof of his theorem, MacPherson defined two important concepts as the main ingredients in his definition of $c_*$: the local Euler obstruction function $Eu_V$ and the Chern-Mather class $c_M(V)$ assigned to any complex variety $V$. Then the proof decomposes into the following parts. Let $X$ be a compact complex variety,
\begin{enumerate}
\item For any subvariety $W\subset X$, $Eu_W$ is a constructible function, i.e., $Eu_W=\sum_{V} e_Z \ind_Z$ for some sub-varieties $Z$ of $X$.
\item $\{Eu_W| W \text{ is a subvariety of } X\}$ form a base for $F(X)$.
\item Define $c_*(Eu_W)=i_*(c_M(W))$ to be the pushforward of the Chern-Mather class of $W$ in $A_*(X)$. This is the unique natural transformation that matches the desired normalization property. 
\end{enumerate}

In the next two sections we will give the definition of the Chern-Mather class and the local Euler obstruction function, together with some basic properties.

\subsection{Nash Blowup and Chern-Mather class}
\label{S; NB}
Let $K$ be any algebraically closed field. Let $M$ be a  smooth ambient variety over $K$, and $X$ be a $n$-dimension subvariety of $M$. Let $X_{sm}\subset X$ be the smooth part of $X$, and define the Gauss map $\alpha\colon X_{sm}\to Gr_n(TM)$ sending $x$ to $T_x X$. Let $S$ be the tautological subbundle on $Gr_n(TM)$. 
\begin{defn}
We define the \textit{Nash blowup} of $X$ to be the closure of $\alpha(X_{sm})$, denoted by $NB$. We also define the restriction of $S$ on $NB$ to be the Nash tangent bundle, denoted by $\calT$.
\end{defn}
One has the following diagram:
\[
\begin{tikzcd}
 NB \arrow{r}{\subset} \arrow{d}{\nu} & Gr_n(TM) \arrow{d}\\
X \arrow{r}{\subset} & M
\end{tikzcd} \/.
\]
Here $\nu$ is the restriction of the projection from $Gr_n(TM)$ to $X$. This is a proper, birational morphism, and restricts to an isomorphism over $X_{sm}$.

\begin{defn}
We define $c_M(X)$, the Chern-Mather class of $X$ to be
\[
c_M(X):=\nu_*(c(\calT)\cap [NB]) \/.
\] 
\end{defn}

\begin{remark} {\quad}\\
\begin{enumerate}
\item The Chern-Mather class of $X$ is independent of the choice of the ambient space $M$, since $\nu$ is made by local embedding and gluing process. Moreover, when $X$ is smooth, 	 $c_M(X)=c(T_X)\cap [X]$ is the total Chern class of $X$. 
\item 
The original definition made by MacPherson was for complex varieties, but one can extend the definition to arbitrary base field $K$. (Cf. \cite{Kennedy}).
\end{enumerate}
\end{remark}

\subsection{Local Euler Obstruction}
\label{S; Euler}
In his paper~\cite{MAC} MacPherson defined the local Euler obstruction of $X$ at $p$ using obstruction theory. This definition relies on the topology of the complex structure of $X$, and to generalize it to arbitrary field we are going to use an equivalent algebraic definition. The algebraic definition was first introduced by Gonz\'alez-Sprinberg and Verdier in \cite{MR629121}, as an integration over the fiber of the Nash Blowup. It agrees with MacPherson's original definition on $\bbC$. 

\begin{defn}
Let $X$ be a complete embeddable variety over field $K$. Let $\nu\colon NB\to X$ and $\calT$ be the Nash Blowup of $X$ and the Nash tangent bundle respectively. 
For any $p\in X$, we define the local Euler obstruction $Eu_V(p)$ to be
\begin{equation}
\label{defn; eulerobstruction}
Eu_V(p)=\int c(\calT)\cap s(\nu^{-1}(p),NB) .
\end{equation}
Here $s(\nu^{-1}(p),NB)$ is the Segre class. 
\end{defn}
\begin{remark}
When $X\to Y$ is a regular embedding of smooth varieties, which is indeed the case in this paper, the Segre class $s(X,Y)$ is the inverse of the Chern class of the normal bundle $\calN_X Y$, i.e., $s(X,Y)=c^{-1}(\calN_X Y)\cap [X]$. For an arbitrary closed embedding, the precise definition and basic properties of Segre class are given in\cite[Chapter 4]{INT}.
\end{remark}

\section{Determinantal Varieties and their Resolutions}
\label{S; determinantalvar}
\subsection{Determinantal Variety}
Let $K$ be an algebraically closed base field. For $m\geq n$, let $M_{m,n}=M_{m,n}(K)$ be the set of $m \times n$ nonzero
matrices over $K$ up to scalar. We view this set as a projective space
 $\bbP^{mn-1}=\bbP(Hom(V_n,V_m))$ for some $n$ dimensional vector space $V_n$ and $m$ dimensional vector space $V_m$ over $K$. 
For $0\leq k\leq n-1$, we consider the subset $\tau_{m,n,k} \subset M_{m,n}$ consisting of all the matrices whose kernel has dimension no less than $k$, or equivalently with rank no bigger than $n-k$. Since the rank condition is equivalent to the vanishing of all $(n-k+1)\times(n-k+1)$ minors, $\tau_{m,n,k}$ is a subvariety of $\bbP^{mn-1}$. The varieties $\tau_{m,n,k}$ are called \textit{(generic) Determinantal Varieties}.

The determinantal varieties have the following basic properties:
\begin{enumerate}
\item When $k=0$, $\tau_{m,n,k}=\bbP^{mn-1}$ is the whole porjective space.
\item When $k=n-1$, $\tau_{m,n,n-1}\cong \bbP^{m-1}\times \bbP^{n-1}$ is isomorphic to the Segre embedding.
\item $\tau_{m,n,k}$ is irreducible, and $\dim\tau_{m,n,k}=(m+k)(n-k)-1$.
\item For $i\leq j$, we have the natural closed embedding $\tau_{m,n,j}\hookrightarrow \tau_{m,n,i}$. Especially, for $j=i+1$, we denote the open subset $\tau_{m,n,i}\smallsetminus \tau_{m,n,i+1}$ by $\tau_{m,n,i}^\circ$.
\item For $k\geq 1$, and $n\geq 3$, the varieties $\tau_{m,n,k}$ are singular with singular locus $\tau_{m,n,k+1}$. Hence $\tau_{m,n,k}^\circ$ is the smooth part of $\tau_{m,n,k}$.
\item For $i=0,1,\cdots ,n-1-k$, the subsets $\tau_{m,n,k+i}^\circ$ form a disjoint decomposition of $\tau_{m,n,k}$. When $K=\bbC$, this is a Whitney stratification of $\tau_{m,n,k}$.
\end{enumerate}

\subsection{The Tjurina Transform}
\label{Tjurina}
In this section we introduce the Tjurina Transform $\hat \tau_{m,n,k}$ as a resolution of $\tau_{m,n,k}$, which will be used later in the computation of the Chern-Mather class. The Tjurina transform is defined to be the incidence correspondence in $G(k,n)\times\bbP^{mn-1}$:
\[
\hat \tau_{m,n,k}:=\{(\Lambda,\varphi)|\varphi\in \tau_{m,n,k};  \Lambda\subset\ker\varphi\}.
\]
And one has the following diagram:
\begin{small}
\[
\begin{tikzcd}
 & \hat\tau_{m,n,k} \arrow{r}{} \arrow{d}{\nu} \arrow[swap]{dl}{\rho} & G(k,n)\times \bbP^{mn-1} \arrow{d} \\
G(k,n)  & \tau_{m,n,k} \arrow{r}{i} & \bbP^{mn-1}.
\end{tikzcd}
\]
\end{small}
As shown in \cite{XP1}, this is a resolution of $\tau_{m,n,k}$, and moreover $\hat\tau_{m,n,k}$ is isomorphic to the projective bundle $\bbP(Q^{\vee m})$ over the Grassmannian $G(k,n)$. Here $Q$ is the universal quotient bundle. Hence we have the following Euler sequence for the tangent bundle of $\hat\tau_{m,n,k}$.
\[
\begin{tikzcd}
0 \arrow{r} & \calO_{\hat\tau_{m,n,k}} \arrow{r} & \rho^*(Q^{\vee m})\otimes \calO_{\hat\tau_{m,n,k}}(1)\arrow{r} & T_{\hat\tau_{m,n,k}}\arrow{r} & \rho^*T_{G(k,n)} \arrow{r} & 0
\end{tikzcd} \/.
\numberthis \label{eulersequence}
\]

Now we show that the Tjurina transform is a small resolution of $\tau_{m,n,k}$. First let's recall the definition of a small resolution:
\begin{defn}
Let $X$ be a irreducible algebraic variety. Let $p\colon Y\to X$ be a resolution of singularities. $Y$ is called a \textit{Small Resolution} of $X$ if for all $i>0$, 
\[
\codim_X \{x\in X| \dim p^{-1}(x)\geq i\} > 2i .
\]
\end{defn}
 
\begin{prop}
\label{prop; small}
The Tjurina transform $\hat \tau_{m,n,k}$ is a small resolution of $\tau_{m,n,k}$.
\end{prop}
\begin{proof}
Define $L_i=\{x\in \tau_{m,n,k}| \dim \nu^{-1}(x)\geq i \}$. We just need to show that $\codim_{\tau_{m,n,k}} L_i >2i$. Notice that for any $p\in \tau_{m,n,j}^\circ\subset \tau_{m,n,k}$, we have
\[
\nu^{-1}(p)=\{(p,\Lambda)|\Lambda\subset \ker p\}\cong G(k,j) \/.
\]
Hence $\dim \nu^{-1}(p)=k(k-j)$ for any $p\in \tau_{m,n,j}^\circ$, and $L_i=\tau_{m,n,s}$, where $s=k+\lfloor\frac{i}{k}\rfloor$.
So we have
\begin{align*}
\codim_{\tau_{m,n,k}} L_i
=& \codim_{\tau_{m,n,k}} \tau_{m,n,s} \\
=& (m+k)(n-k)-(m+s)(n-s) \\
=& (m+k)(n-k)-(m+k+\lfloor\frac{i}{k}\rfloor)(n-k+\lfloor\frac{i}{k}\rfloor) \\
=&  \lfloor\frac{i}{k}\rfloor((m+k)-(n-k)+\lfloor\frac{i}{k}\rfloor) \\
=&  \lfloor\frac{i}{k}\rfloor(m-n+2k+\lfloor\frac{i}{k}\rfloor) \\
>& 2i \/.
\end{align*}
\end{proof}

When the base field is $\bbC$, $\hat\tau_{m,n,k}$ is the Tjurina transform used in \cite{Tju}. The Tjurina transforms for general determinantal varieties over $\bbC$ is discussed in \cite{MOL}.

\subsection{Nash Blowup of Determinantal Variety}
\label{S; nash}
In this section we describe the Nash Blowup of the variety $\tau_{m,n,k}$ and the Nash tangent bundle. It is the key to our computation of the local Euler obstruction of $\tau_{m,n,k}$.

In \cite[Section 1]{MR2723944} \`Ebeling and Guse\u\i n-Zade constructed the Nash blowup of the affine complex determinantal varieties, i.e., the affine cone $\Sigma_{m,n,k}\subset \bbC^{mn}$ of $\tau_{m,n,k}$. For $\tau_{m,n,k}$ over arbitrary field $K$, we have the following similar result. Define $N_{m,n,k}$ to be the incidence correspondence in $G(k,n)\times G(n-k,m)\times\bbP^{mn-1}$ :
\[
N_{m,n,k}:= \{(\Lambda,\Gamma,\varphi)|\varphi\in \tau_{m,n,k}; \Lambda\subset\ker\varphi;   \Ima\varphi\subset\Gamma\} .
\]
Let $p$ be the projection on $G(k,n)\times G(n-k,m)$, and $\pi$ be the projection on $\bbP^{mn-1}$. One has the following diagram: 
\begin{small}
\[
\begin{tikzcd}
 & N_{m,n,k} \arrow{r}{} \arrow{d}{\pi} \arrow[swap]{dl}{p} & G(k,n)\times G(n-k,m)\times \bbP^{mn-1} \arrow{d} \\
G(k,n)\times G(n-k,m)  & \tau_{m,n,k} \arrow{r}{i} & \bbP^{mn-1}.
\end{tikzcd}
\]
\end{small}

\begin{prop}
Let $(S_1,Q_1)$ and $(S_2,Q_2)$ be the universal subbundles and quotient bundles over $G(k,n)$ and $G(n-k,m)$ respectively.
The variety $N_{m,n,k}$ can be identified with the projective bundle $\bbP(Q_1^{\vee}\otimes S_2)$ over the Grassmannian $G(k,n)\times G(n-k,m)$, hence it is smooth. Moreover, $N_{m,n,k}$ is the Nash Blowup of $\tau_{m,n,k}$.
\end{prop}
\begin{proof}
For any linear map $\varphi\colon K^n\to \Gamma\subset K^m$ whose kernel contains $\Lambda$   one to one corresponds to a linear map $K^n/\Lambda\to \Gamma$. We have the following identification:
\[
N_{m,n,k}\cong \bbP(Hom(Q_1,S_2))=\bbP(Q_1^{\vee}\otimes S_2).
\]
Hence $N_{m,n,k}$ is a smooth projective variety. Also, $\pi$ is one to one over $\tau_{m,n,k}^\circ$, and the dimension of $N_{m,n,k}$ is 
\[
k(n-k)+(n-k)(m-n+k)+(n-k)^2-1=(m+k)(n-k)-1=\dim\tau_{m,n,k},
\]
Hence $\pi\colon N_{m,n,k}\to \tau_{m,n,k}$ is a resolution of $\tau_{m,n,k}$.

Let $\Sigma_{m,n,k}=\{A\in Hom(K^n,K^m)|\dim\ker A\geq k\}\subset \bbA^{mn}$ be the affine cone of $\tau_{m,n,k}$. It is also generally singular for $k\geq 1$, with singular locus $\Sigma_{m,n,k+1}$. We denote its smooth part by $\Sigma_{m,n,k}^\circ:=\Sigma_{m,n,k}\smallsetminus \Sigma_{m,n,k+1}$. From \cite[Lecture 14]{AG-JH} we know that, for any $A\in \Sigma_{m,n,k}^\circ$, the tangent space is identified with
\[
T_A \Sigma_{m,n,k}^\circ =\{B\in Hom(K^n,K^m)| B(\ker A) \subset \Ima A\}.
\]
And for any $\varphi\in \tau_{m,n,k}^\circ$, its tangent space is
\[
T_\varphi \tau_{m,n,k}^\circ =(T_A \Sigma_{m,n,k}^\circ/[\varphi])\otimes [\varphi]^*
\]
for any $A\neq 0\in [\varphi]$. Her $[\varphi]=\calO(-1)|_\varphi$ denotes the line represented by $\varphi$.

Let $d=(m+k)(n-k)-1=\dim \tau_{m,n,k}$. By definition the Nash Blowup of $\tau_{m,n,k}$ is the closure of the Gauss map: 
\[
\alpha\colon \tau_{m,n,k}^\circ \to G_d T\bbP^{mn-1}; \quad \varphi\mapsto T_\varphi \tau_{m,n,k}^\circ \/.
\]
We define the following morphisms: 
\[
\alpha' \colon \tau_{m,n,k}^\circ \to N_{m,n,k}; \quad  \varphi\mapsto (\varphi, \ker\varphi, \Ima\varphi)
\]
and 
\[
\beta \colon N_{m,n,k} \to G_d T\bbP^{mn-1}
\]
sending
\[ 
(\varphi, \Lambda, \Gamma)\mapsto (\varphi, \{B\colon  K^n\to K^m | B(\Lambda) \subset \Gamma\}/[\varphi]\otimes [\varphi]^*) \/.
\]
The morphisms fit into the  commutative diagram:
\[
\begin{tikzcd}
\tau_{m,n,k}^\circ \arrow{r}{\alpha'}\arrow{dr}{\alpha} &  N_{m,n,k} \arrow{d}{\beta}\\
&  G_d (T\bbP^{mn-1})  
\end{tikzcd} \/.
\]

To see that $N_{m,n,k}$ is indeed the Nash Blowup, one just need to show that 
\[
\overline{\beta\alpha'(\tau_{m,n,k}^\circ)} =\overline{\alpha(\tau_{m,n,k}^\circ)} \/,
\]
and $\beta$ is a closed embedding.

For the first property, notice that $N_{m,n,k}$ and $G_d (T\bbP^{mn-1})$ are both projective, $\beta$ is then proper.  Then since $\dim \tau_{m,n,k}=\dim N_{m,n,k}$, $\alpha'(\tau_{m,n,k}^\circ)$ is dense in $N_{m,n,k}$, and 
\[
\overline{\beta\alpha'(\tau_{m,n,k}^\circ)}=\beta\overline{\alpha'(\tau_{m,n,k}^\circ)}
=\beta(N_{m,n,k})\cong N_{m,n,k} \/.
\] 

To show that $\beta$ is a closed embedding, we show that it is injective and separates tangent vectors.
To see the injectivity, one just need to show $ \{B|B(\Lambda) \subset \Gamma\}\neq \{B|B(\Lambda') \subset \Gamma'\} $ whenever $(\Lambda, \Gamma)\neq (\Lambda', \Gamma')$. Assume $\Lambda\neq \Lambda'$, then there exist $v,w$ such that $v\in\Lambda$, $w\notin\Gamma$, and $v\notin\Lambda'$. Define $B$ to map $v\mapsto w$ and everything else to $0$, then $B(\Lambda') \subset \Gamma'$, but $B(\Lambda) \not\subset \Gamma$. Similar construction applies to $\Gamma\neq \Gamma'$ case.

In terms of the tangent vectors, for the point $(\varphi, \Lambda, \Gamma)$, we take an affine neighborhood $\bbA^{mn-1}$ of $\varphi \in \bbP^{mn-1}$, then the Grassmannian bundle $G_d T\bbP^{mn-1}$ restricts on $\bbA^{mn-1}$ is a global product $\bbA^{mn-1}\times G(d,mn-1)$. Let $U=N_{m,n,k}\cap \bbA^{mn-1}\times G(k,n)\times G(n-k,m)$, then $\beta|_U\colon U\to\bbA^{mn-1}\times G(d,mn-1)$ is a product $1\times \beta'$. Hence it would be enough to show that $d\beta'\colon T_{(\Lambda, \Gamma)}G(k,n)\times G(n-k,m) \to T_{W}G(d,mn-1)$ is injective, where $W=\{B|B(\Lambda)\subset \Gamma\}/L$ for some fixed line $L\subset \bbA^{mn}$. Let $\Omega=\{\Lambda'|L\subset \Lambda'\}\subset G(d+1,mn)$, then $\pi\colon \Omega\to G(d,mn-1)$ sending $\Lambda'\mapsto \Lambda'/L$ is an isomorphism.  Also, notice that $\rho\colon G(k,n)\times G(n-k,m)\to G(d+1,mn)$ sending $(\Lambda, \Gamma)\mapsto \{B|B(\Lambda)\subset \Gamma\}$ is a closed embedding. The injectivity of $d\beta'$ comes from the observation that $\beta'$ factors through $\rho$ and $\pi$ locally. 
\end{proof}

We also have the following result for the Chern class of the Nash tangent bundle $\calT$:
\begin{lemma}
\label{lemma; nashtangent}
With the above notations, the Chern class of the Nash tangent bundle $\calT$ over $N_{m,n,k}$ equals:
\[
c(\calT)=c(\calO(1))^{mn}c^{-1}(S_1^\vee\otimes Q_2\otimes\calO(1)) \/.
\]
\end{lemma}
\begin{proof}
Let $d=\dim \tau_{m,n,k}=(m+k)(n-k)-1$. From \cite[Lecture 14]{AG-JH} we can construct the following map:
\[
\alpha\colon  G(k,n)\times G(n-k,m)\to  G(d+1,\bbA^{mn})
\]
sending $(\Lambda,\Gamma)$ to $ \{B\colon  K^n\to K^m | B(\Lambda) \subset \Gamma\}$. 
The universal sub-bundle of $G(d+1,\bbA^{mn})$ restricts to a vector bundle of rank $d+1$ 
on $G(k,n)\times G(n-k,m)$, and we denote it by $E$. 
One can verify the following exact sequence 
\[
\begin{tikzcd}
0 \arrow{r} & E\arrow{r} & Hom(K^n, K^m)  \arrow{r} & Hom(S_1, Q_2) \arrow{r} & 0
\end{tikzcd} \/.
\]

For the variety $\tau_{m,n,k}$, the morphism $\beta \colon N_{m,n,k} \to G_d T\bbP^{mn-1}$ sends $(\varphi, \Lambda, \Gamma)$ to $(\{B\colon  K^n\to K^m | B(\Lambda) \subset \Gamma\}/[\varphi])\otimes [\varphi]^*$. Hence
over $N_{m,n,k}$ the Nash tangent bundle $\calT$ fits in the following exact sequence:
\[
\begin{tikzcd}
0 \arrow{r} & \calO(-1) \arrow{r} &  p^*E \arrow{r} & \calT\otimes \calO(-1) \arrow{r} & 0
\end{tikzcd} \/.
\]

Then we have
\begin{align*}
c(\calT)
&=c(\pi^*E\otimes \calO(1)) \\
&=c(Hom(K^n, K^m)\otimes \calO(1) )c^{-1}(Hom(S_1, Q_2)\otimes \calO(1)) \\
&=c(\calO(1))^{mn}c^{-1}(S_1^\vee\otimes Q_2\otimes\calO(1)) .
\end{align*}
\end{proof}

\section{Local Euler Obstruction of $\tau_{m,n,k}$}
In this section we compute the local Euler obstruction of $\tau_{m,n,k}$. 
When the base field is $\bbC$, the formula for the local Euler obstruction of $\Sigma_{m,n,k}$ was  found by N. Grulha, T. Gaffney and M. Ruas in \cite{NG-TG}. 
In the paper they found the numerical formula of the local Euler obstruction using topology methods, and observed the property of the Pascal's triangle \cite[Figure 1]{NG-TG}. In this paper we use a direct intersection-theoretic approach and generalize the formula for the local Euler obstruction of $\tau_{m,n,k}$ to
any algebraically closed base field $K$.
\begin{thm}
\label{thm; eulerobstruction}
Let $K$ be any algebraically closed base field. Let $\tau_{m,n,k}$ be the determinantal variety over $K$. For any $\varphi\in\tau_{m,n,k+i}^\circ$, the local Euler obstruction of $\tau_{m,n,k}$ at $\varphi$ is:
\begin{align*}
Eu_{\tau_{m,n,k}}(\varphi)
&=\int_{G(k,k+i)\times G(i,m-n+k+i)} c^{-1}(S_1^\vee\otimes Q_2)c^{-1}(Q_1^\vee\otimes S_2) \\
&=\int_{G(k,k+i)\times G(i,m-n+k+i)} c_{top}(S_1^\vee\otimes S_2)c_{top}(Q_1^\vee\otimes Q_2) \/.
\end{align*}
\end{thm}
\begin{proof}
According to the formula \eqref{defn; eulerobstruction} by Gonzalez-Springer and Verdier, the local Euler obstruction of $\tau_{m,n,k}$ at $\varphi$ is:
\[
Eu_{\tau_{m,n,k}}(\varphi)=\int_{\pi^{-1}(\varphi)}c(\calT)s(\pi^{-1}(\varphi), N_{m,n,k}) \/.
\]
For any $\varphi\in\tau_{m,n,k+i}^\circ$, the fiber over $\varphi$ is
\[
\pi^{-1}(\varphi)=\{(\Lambda,\Gamma)|\Lambda \subset \ker \varphi; \Ima \varphi\subset \Gamma\}\cong\{\varphi\}\times G(k,k+i)\times G(i, m-n+k-i) \/.
\]
Since the fiber $\pi^{-1}(\varphi)$ and $N_{m,n,k}$ are both smooth, $\pi^{-1}(\varphi)\to N_{m,n,k}$ is a regular embedding, and we have
\[
s(\pi^{-1}(\varphi), N_{m,n,k})=c^{-1}(N_{\pi^{-1}(\varphi)}N_{m,n,k})=\frac{c(T_{\pi^{-1}(\varphi)})}{c(T_{N_{m,n,k}})} \/.
\] 
Since $N_{m,n,k}$ is a projective vector bundle $\bbP(Q_1^{\vee}\otimes S_2)$ over $G(k,n)\times G(n-k,m)$, we have the following Euler sequence 
\[
\begin{tikzcd}
0 \arrow{r} & \calO_{N_{m,n,k}} \arrow{r} & p^*(Q_1^{\vee}\otimes S_2)\otimes \calO(1))\arrow{r} & T_{N_{m,n,k}}\arrow{r} & p^*T_{G(k,n)\times G(n-k,m)} \arrow{r} & 0 
\end{tikzcd} \/.
\]
Hence we have:
\[
c(T_{N_{m,n,k}})=c(p^*(Q_1^{\vee}\otimes S_2)\otimes \calO(1)))\cdot c(p^*T_{G(k,n)\times G(n-k,m)}) \/.
\]

We use $S_i'$ and $Q_i'$, $i=1,2$ respectively to denote the pull back of universal subbundles and quotient bundles of $G(k,n)\times G(n-k,m)$ on $\nu^{-1}(\varphi)$. One has:
\begin{align*}
s(\pi^{-1}(\varphi), N_{m,n,k})
&~= \frac{c(T_{\pi^{-1}(P)})}{c(T_{N_{m,n,k}})} \cap [\pi^{-1}(\varphi)] \\
&~= \frac{c(S_1^\vee\otimes Q_1)c(S_2^\vee\otimes Q_2)}{c(Q_1^{'\vee}\otimes S'_2\otimes \calO(1))c(S_1^{'\vee}\otimes Q'_1)c(S_2^{'\vee} \otimes Q'_2)} \cap [\pi^{-1}(\varphi)] \\
&~= \frac{c(Q_1^\vee\otimes S_2\otimes \calO(1))^{-1}c(Q_1^\vee)^{-n+k+i}c(S_2)^{-n+k+i}}{c^{n-k-i}(S_1^\vee)c^{n-k-i}(Q_2) } \cap [\pi^{-1}(\varphi)] \\
&~= c(Q_1^\vee\otimes S_2\otimes \calO(1))^{-1} \cap [\pi^{-1}(\varphi)] \/.
\end{align*}

So the local Euler obstruction is 
\begin{align*}
Eu_{\tau_{m,n,k}}(\varphi)
&~= \int_{\pi^{-1}(\varphi)}c(\calT)s(\pi^{-1}(\varphi), N_{m,n,k}) \\
&~= \int_{\{\varphi\}\times G(k,k+i)\times G(i,m-n+k+i)} c(\calT)c(Q_1^\vee\otimes S_2\otimes \calO(1))^{-1} \\
&~= \int_{G(k,k+i)\times G(i,m-n+k+i)} \frac{c(\calO(1))^{mn}}{c(S_1^\vee\otimes Q_2\otimes\calO(1))c(Q_1^\vee\otimes S_2\otimes \calO(1)) } \/.
\end{align*}
Since $\pi^{-1}(\varphi)=\{\varphi\}\times G(k,k+i)\times G(i, m-n+k-i) $, the tautological line bundle $\calO(1)$ restricts to the trivial line bundle over $\pi^{-1}(\varphi)$, hence we have

\begin{align*}
Eu_{\tau_{m,n,k}}(\varphi)= \int_{G(k,k+i)\times G(i,m-n+k+i)} c(S_1^\vee\otimes Q_2)^{-1}c(Q_1^\vee\otimes S_2)^{-1} \/.
\end{align*}
To verify the second formula, use the following exact sequences:
\[
\begin{tikzcd}
0 \arrow{r} & S_1^\vee\otimes S_2 \arrow{r} &  S_1^\vee\otimes K^m  \arrow{r}&  S_1^\vee\otimes Q_2 \arrow{r} & 0 \\
0 \arrow{r} & Q_1^\vee\otimes S_2 \arrow{r} &  Q_1^\vee\otimes K^m  \arrow{r}&  Q_1^\vee\otimes Q_2 \arrow{r} & 0 . \\
\end{tikzcd} \/.
\]
One can get 
\begin{align*}
c(S_1^\vee\otimes Q_2)^{-1}c(Q_1^\vee\otimes S_2)^{-1}
&=\frac{c(S_1^\vee\otimes S_2)}{c(S_1^\vee\otimes K^m)}
 \cdot \frac{c(Q_1^\vee\otimes Q_2)}{c(Q_1^\vee\otimes K^m)} \\
&=  c(S_1^\vee\otimes S_2)c(Q_1^\vee\otimes Q_2) \\
&= c(S_1^\vee\otimes S_2\oplus Q_1^\vee\otimes Q_2)  \/.
\end{align*}
Since the rank of $S_1^\vee\otimes S_2\oplus Q_1^\vee\otimes Q_2$ equals $k(i)+(i)(m-n+k+i)=\dim G(k,k+i)\times G(i,m-n+k+i)$, we have:
\begin{align*}
Eu_{\tau_{m,n,k}}(\varphi)
&= \int_{G(k,k+i)\times G(i,m-n+k+i)} c(S_1^\vee\otimes Q_2)^{-1}c(Q_1^\vee\otimes S_2)^{-1} \\
&= \int_{G(k,k+i)\times G(i,m-n+k+i)} c(S_1^\vee\otimes S_2)c(Q_1^\vee\otimes Q_2) \\
&= \int_{G(k,k+i)\times G(i,m-n+k+i)} c_{top}(S_1^\vee\otimes S_2)c_{top}(Q_1^\vee\otimes Q_2) \/.
\end{align*}
\end{proof}

\begin{ex}
When $k=0$, $\tau_{m,n,0}=\bbP^{mn-1}$ is the projective space. For any $\varphi\in \tau_{m,n,i}^\circ$, we have
\begin{align*}
Eu_{\tau_{m,n,0}}(\varphi)
&= \int_{G(0,i)\times G(i,m-n+i)} c_{top}(S_1^\vee\otimes S_2)c_{top}(Q_1^\vee\otimes Q_2) \/.
\end{align*}
Since $G(0,i)=\{pt\}$, $S_1=0$ and $Q_1=K^i$. Hence 
\begin{align*}
Eu_{\tau_{m,n,0}}(\varphi)
&= \int_{G(0,i)\times G(i,m-n+i)} c_{top}(S_1^\vee\otimes S_2)c_{top}(Q_1^\vee\otimes Q_2) \\
&= \int_{G(i,m-n+i)} c_{m-n}(Q)^i  \/.
\end{align*}
Over the Grassmannian $G(k,n)=G(k,V_n)$, let $\lambda=(n-k,n-k,\cdots ,n-k)$ be a partition, and $\underline{a}=(1,2,\cdots ,k)$ be a $k$ tuple. Let $c=c(Q)$, and define $A_i\subset V_n$ to be the subspace generated by first $i$ coordinates. 
The Giambelli's formula (see \cite[\S 14.6]{INT}) shows that
\begin{align*}
c_{n-k}(Q)^k\cap [G(k,n)]
=& \Delta_{\lambda}(c) \cap [G(k,n)] \\
=& [\Omega(\underline{a})] \\
=& [\{\Lambda\in G(k,n)|\text{ for any i, } \dim(\Lambda\cap A_i)\geq i\} ]\\
=& [\{A_k\}] = [pt] \/.
\end{align*}
Hence 
\[
Eu_{\tau_{m,n,0}}(\varphi)=\int_{G(i,m-n+i)} c_{m-n}(Q)^i =\int_{G(i,m-n+i)} [pt]=1 \/.
\]
Which is indeed the case since $\tau_{m,n,0}$ is smooth.
\end{ex}

We denote $Eu_{\tau_{m,n,k}}(\tau_{m,n,k+i}^\circ):=Eu_{\tau_{m,n,k}}(\varphi)$ for  any $\varphi\in \tau_{m,n,k+i}^\circ$ to be the local Euler obstruction of $\tau_{m,n,k}$ on the stratum $\tau_{m,n,k+i}^\circ$. From the theorem one can see that
\begin{align*}
Eu_{\tau_{m,n,k}}(\tau_{m,n,k+i}^\circ)
&=Eu_{\tau_{m-1,n-1,k}}(\tau_{m-1,n-1,k+i}^\circ)\\
&=Eu_{\tau_{m-2,n-2,k}}(\tau_{m-2,n-2,k+i}^\circ)\\
& \cdots \\
&=Eu_{\tau_{m-n+k+i+1,k+i+1,k}}(\tau_{m-n+k+i+1, k+i+1,k+i}) \/.
\end{align*}

This shows that it would be enough to compute the local Euler obstruction of $\tau_{m,n,k}$ on the smallest stratum $\tau_{m,n,n-1}$. We define $e(m,n,k):=Eu_{\tau_{m,n,k}}(\tau_{m,n,n-1})$ as a function of $m,n,k$. 
\begin{thm}[Pascal's Triangle] 
\label{thm; pascal}
We have the following recursive formula for $e(m,n,k)$, $k=0,1,\cdots ,n-2$:
\[
e(m,n,k)+e(m,n,k+1)=e(m+1,n+1,k+1) \/.
\]
Hence we have 
\[
e(m,n,k)=\binom{n-1}{k}
\]and
\[
Eu_{\tau_{m,n,k}}(\tau_{m,n,k+i}^\circ)=\binom{k+i}{i} \/.
\]
\end{thm}
\begin{proof}
We need to prove the following equation:
\begin{align*}
& \int_{G(k+1,n+1)\times G(n-k,m+1)}  c_{top}(S_1^\vee\otimes S_2)c_{top}(Q_1^\vee\otimes Q_2) \\
=&
\int_{G(k,n)\times G(n-k,m)}  c_{top}(S_1'^\vee\otimes S_2')c_{top}(Q_1'^\vee\otimes Q_2') \\
+& \int_{G(k+1,n)\times G(n-k-1,m)}  c_{top}(S_1''^\vee\otimes S_2'')c_{top}(Q_1''^\vee\otimes Q_2'') .
\end{align*}
Here for $i=1,2$ we denote $S_i',Q_i'$, $S_i'',Q_i''$ and $S_i,Q_i$ the universal subbundles and quotient bundles over $G(k,n)\times G(n-k,m)$, $G(k+1,n)\times G(n-k-1,m)$ and $G(k+1,n+1)\times G(n-k,m+1)$ respectively.

To compare the classes over different Grassmannians, we need to embed the small Grassmannians into big ones.
Let $V=V_n$ be the $n$-dimensional $K$-vector space, and $G(k,V_n)=G(k,n)$ to be the Grassmannian. We define the following embeddings:
\begin{enumerate}
\item[*] $i\colon G(k,V_n)\to G(k,V_n\oplus 1)$ sending the column space $\Lambda=[v_1;v_2;\cdots ;v_k]$ to $[v_1\oplus 0;v_2\oplus 0;\cdots ; v_k\oplus 0]$.
\item[*] $j\colon G(k,V_n)\to G(k+1,V_n\oplus 1)$ sending the column space $\Lambda=[v_1;v_2;\cdots ; v_k]$ to $[v_1\oplus 0;v_2\oplus 0;\cdots ;v_k\oplus 0;v_{k+1}]$. Here $v_{k+1}=0\oplus 1=(0,0,\cdots ,0, 1)$.
\end{enumerate}
Let $S_{k,n},Q_{k,n}$ be the universal sub and quotient bundles over $G(k,n)$, let $1$ be the trivial bundle of rank $1$. One has the following observations:
\begin{align*}
i^*S_{k,n+1}=S_{k,n} , & \quad  i^*Q_{k,n+1}=Q_{k,n}\oplus 1 \\
j^*S_{k+1,n+1}=S_{k,n}\oplus 1 , & \quad j^*Q_{k+1,n+1}=Q_{k,n}
\end{align*}
and
\begin{align*}
i_*([G(k,n)])=c_k(S^\vee_{k.n+1})\cap [G(k,n+1)] \\
j_*([G(k,n)])=c_{n-k}(Q_{k+1,n+1})\cap [G(k+1,n+1)] \/.
\end{align*}

So in terms of the product of Grassmannians $M:=G(k+1,n+1)\times G(n-k,m+1)$, 
we define the closed embeddings
$u=(j,i)\colon M_1:=G(k,n)\times G(n-k,m)\to M$ and 
$v=(i,j)\colon M_2:=G(k+1,n)\times G(n-k-1,m)\to M$.

And we have:
\[
u_*[M_1]=c_{top}(Q_1)c_{top}(S_2^\vee)\cap [M] ;\quad v_*[M_2]=c_{top}(S_1^\vee)c_{top}(Q_2)\cap [M] \/.
\]

We are going to use the Chern roots to prove the above equation of Chern classes. If $E$ is a rank $k$ vector bundle over a variety
$X$, we say that a flat morphism $f\colon X’ \to X$, along with a complete filtration $0 \subset E_1 \subset … \subset E_k=f^*E$
is a `splitting morphism of E’ if $f^*\colon
A_*X \to A_*X’$ is injective. Splitting morphisms may be constructed as a sequence
of projective bundles (cf. \cite[\S 3.2]{INT}).  Let $\alpha_i=c_1(E_i/E_{i-1})$ for $i=1,2,\cdots ,k$, then $c(f^*E)=\prod_i (1+\alpha_i)$.
Moreover, by the injectivity of $f^*$, equations involving the Chern classes of $E$ may be
proven by proving corresponding equations for symmetric functions in $\{\alpha_1,\alpha_2,\cdots , \alpha_k\}$. We call the classes $\alpha_i:=c_1(E_i/E_{i-1})$ the Chern roots of $E$.

Let the Chern roots of $S_1,Q_1$ be $\{\alpha_1,\alpha_2,\cdots ,\alpha_{k+1}\}$ and $\{\beta_1,\beta_2,\cdots ,\beta_{n-k}\}$, and the Chern roots of $S_2,Q_2$ be $\{a_1,a_2,\cdots ,a_{n-k}\}$ and $\{b_1,b_2,\cdots ,b_{m-n+k+1}\}$. Then we have
\begin{align*}
& c_{top}(S_1^\vee\otimes S_2)c_{top}(Q_1^\vee\otimes Q_2) \\
=&
\prod_{i=1}^{n-k}\prod_{j=1}^{k+1} (a_i-\alpha_j)\prod_{i=1}^{m-n+k+1}\prod_{j=1}^{n-k} (b_i-\beta_j) \\
=
&  (a_1-\alpha_1)(a_1-\alpha_2) \cdots (a_1-\alpha_{k+1}) \\
&  (a_2-\alpha_1)(a_2-\alpha_2) \cdots (a_2-\alpha_{k+1}) \\
& \cdots  \\
&  (a_{n-k}-\alpha_1)(a_{n-k}-\alpha_2) \cdots  (a_{n-k}-\alpha_{k+1}) \\
&  (b_1-\beta_1)(b_1-\beta_2) \cdots (b_1-\beta_{n-k}) \\
&  (b_2-\beta_1)(b_2-\beta_2) \cdots  (b_2-\beta_{n-k}) \\
& \cdots  \\
&  (b_{m-n+k+1}-\beta_1)(b_{m-n+k+1}-\beta_2) \cdots  (b_{m-n+k+1}-\beta_{n-k}) \\
=:&  P \/.
\end{align*}

\begin{lemma}
Every term in the expansion of $P$ is a multiple of one of the following:
\begin{align*}
a_1a_2\cdots a_{n-k}\beta_1\beta_2\cdots \beta_{n-k} \\
\alpha_1\alpha_2\cdots\alpha_{k+1}b_1b_2\cdots b_{m-n+k+1} \/.
\end{align*}
\end{lemma}
\begin{proof}
If one of the term doesn't have $a_1a_2\cdots a_{n-k}$, assume that $a_i$ is missing, then we consider the row $(a_i-\alpha_1)(a_i-\alpha_2) \cdots  (a_i-\alpha_{k+1})$. Since there is no $a_i$, the term will have to contain $\alpha_1\alpha_2\cdots\alpha_{k+1}$. This analogous observation applies to $b_1b_2\cdots b_{m-n+k+1}$ and $\beta_1\beta_2\cdots \beta_{n-k}$.

Hence every term will be a multiple of one of the following:
\begin{enumerate}
\item $\alpha_1\alpha_2\cdots\alpha_{k+1}b_1b_2\cdots b_{n-k}$
\item $\alpha_1\alpha_2\cdots\alpha_{k+1}\beta_1\beta_2\cdots \beta_{n-k}$
\item $a_1a_2\cdots a_{n-k}\beta_1\beta_2\cdots \beta_{n-k}$
\item $a_1a_2\cdots a_{n-k}b_1b_2\cdots b_{m-n+k+1}$ \/.
\end{enumerate}
The lemma comes from the following vanishing property:
\begin{align*}
\alpha_1\alpha_2\cdots\alpha_{k+1}\beta_1\beta_2\cdots \beta_{n-k}&=c_{top}(S_1)c_{top}(Q_1)=c_{top}(K_{n+1})=0 \\
a_1a_2\cdots a_{n-k}b_1b_2\cdots b_{m-n+k+1}&=c_{top}(S_2)c_{top}(Q_2)=c_{top}(K_{m+1})=0 .
\end{align*}
\end{proof}

Hence we can separate $P$ into two parts:
\[
P=A(-\alpha_1)(-\alpha_2)\cdots(-\alpha_{k+1})b_1b_2\cdots b_{m-n+k+1}+Ba_1a_2\cdots a_{n-k}(-\beta_1)(-\beta_2)\cdots (-\beta_{n-k}) .
\]
Here $A$ and $B$ are symmetry functions of $\{\alpha_i; \beta_i; a_i; b_i\}$, hence they can be expressed as polynomials of Chern classes of $S_i,Q_i$, and hence one can pull them back through morphisms.

Notice that 
\begin{align*}
(-\alpha_1)(-\alpha_2)\cdots(-\alpha_{k+1})b_1b_2\cdots b_{m-n+k+1}=& c_{top}(S_1^\vee)c_{top}(Q_2)\\
a_1a_2\cdots a_{n-k}(-\beta_1)(-\beta_2)\cdots (-\beta_{n-k})=& c_{top}(Q_1)c_{top}(S_2^\vee) \/.
\end{align*}
Hence with the notations we defined above, one has
\begin{align*}
\int_{M}  P =
& \int_{G(k+1,n+1)\times G(n-k,m+1)}  c_{top}(S_1^\vee\otimes S_2)c_{top}(Q_1^\vee\otimes Q_2) \\
=&  \int_{M}  Ac_{top}(S_1^\vee)c_{top}(Q_2)+Bc_{top}(Q_1)c_{top}(S_2^\vee) \\
=&  \int_{M_1} u^*(B) + \int_{M_2}  v^*(A) \/.
\end{align*}

Now we just need to prove that 
\[
u^*(B)=c_{top}(S_1^{'\vee}\otimes S'_2)c_{top}(Q_1^{'\vee}\otimes Q'_2)
\]
and
\[
 v^*(A)=c_{top}(S_1^{''\vee}\otimes S^{''}_2)c_{top}(Q_1^{''\vee}\otimes Q^{''}_2) \/.
\]
We only prove the first part, the argument for the second part is analogous. We will need the following lemma:
\begin{lemma}
Let $i\colon G(k,n)\to G(k,n+1)$ be the embedding defined above, and  let $Q'$ and $Q$ to be the universal quotient bundles over 
$G(k,n)$ and $G(k,n+1)$ respectively. Let $\pi\colon M\to G(k,n+1)$ and $0\subset Q_1\subset Q_2\subset \cdots \subset Q_{n-k+1}=\pi^*Q$ be a splitting morphism of $Q$, then there exists a splitting morphism $q\colon M'\to G(k,n)$ and $0\subset Q'_1\subset Q'_2 \subset \cdots \subset Q'_{n-k}=q^*Q'$, together with a closed embedding $f\colon M'\to M$ such that 
\begin{enumerate}
\item $f^*Q_1$ is a trivial line bundle.
\item $f^*c(Q_{i+1}/Q_i)=c(Q'_{i}/Q'_{i-1})$ are the Chern roots of $Q'$. 
\end{enumerate}
The similar statement holds true for the closed embedding $j\colon G(k,n)\to G(k+1,n+1)$, and the universal sub-bundles $S'$ and $S$ defined above.
\end{lemma}
\begin{remark}
Let $\{\beta_{i}=c(Q_{i}/Q_{i-1})\}$ be the Chern roots of $Q$, then by a rearrangement of order we just say that $\{\beta_1,\beta_2,\cdots ,\beta_{n-k}\}$ `pull back' to the Chern roots of $Q'$ by $i$, and $\beta_{n-k+1}$ pulls back to $0$ on $G(k,n)$. 
Also, since $i^*S=S'$, we can say that the Chern roots of $S$ pull back to the Chern roots of $S'$. 
\end{remark}
\begin{proof}
Since the splitting morphism of $Q$ on $G(k,n+1)$ pulls back to a splitting morphism of $i^*Q$ on $G(k,n)$, we just need to prove the statement for $Q'$ and $i^*Q$ on $G(k,n)$. From the previous discussion we know that $i^*Q=Q'\oplus K$. 
On the construction of the splitting morphism, we consider the following diagram
\[
\begin{tikzcd}
\bbP(Q') \arrow{r}{g} \arrow{d}& \bbP(i^*Q)=\bbP(Q'\oplus 1) \arrow{dl}{p}\\
 G(k,n)
\end{tikzcd}
\]  
where $g$ is a closed embedding. Since $g^*(\calO_{\bbP(i^*Q)}(-1))=\calO_{\bbP(Q')}(-1)$, we will write $\calO(-1)$ for short. 
The universal sequences on $\bbP(Q')$ and $\bbP(i^*Q)$ form the following diagram on $\bbP(Q')$:
\[
\begin{tikzcd}
           &  0    \arrow{d}               & 0           \arrow{d} &  0  \arrow{d}       &  \\
0\arrow{r} & \calO(-1) \arrow{r}\arrow{d} & D_{n-k}:=q^*Q' \arrow{r}\arrow{d} & D_{n-k-1} \arrow{r}\arrow{d} & 0 \\
0\arrow{r} & \calO(-1) \arrow{r}\arrow{d} & g^*E_{n-k+1}:=g^*\bbP(i^*Q) \arrow{r}\arrow{d} & g^*E_{n-k} \arrow{r}\arrow{d} & 0  \\
           &  0                   & K                       &  K                   &  
\end{tikzcd} \/.
\]
This shows that 
\[
\frac{c(D_{n-k})}{c(D_{n-k-1})}=g^*\frac{c(E_{n-k+1})}{c(E_{n-k})}.
\]
Then by induction, same argument on $\bbP(D_i)\subset \bbP(E_{i+1})$ shows	
\[
\frac{c(D_i)}{c(D_{i-1})}=g^*\frac{c(E_{i+1})}{c(E_{i})}
\]
when $i=2,3,\cdots ,n-k$. \\
When $i=1$, we have a rank $2$ vector bundle $E_2\subset i^*Q$, and a line bundle $D_1\subset Q'$ with the exact sequence
\[
\begin{tikzcd}
0\arrow{r} & D_1 \arrow{r} & E_2 \arrow{r}  & K \arrow{r} & 0 \/.
\end{tikzcd}
\]
To filter $E_2$, we construct $l\colon \bbP(D_1)\subset \bbP(E_2)$, and filter $E_2$ on $\bbP(E_2)$ by 
\[
\begin{tikzcd}
0\arrow{r} & \calO(-1) \arrow{r} & E_2 \arrow{r}  & E_1 \arrow{r} & 0 \/.
\end{tikzcd}
\]  
Notice that $\calO(-1)|_{\bbP(D_1)}$ is the trivial bundle, hence over the splitting $\bbP(D_1)$, $c_1(\calO(-1))=0$. 
\end{proof}

Hence according to the above lemma, we have 
\begin{align*}
u^*(B)
=& B \text{ evaluated  on } \alpha_{k+1}=0 \text{ and } b_{m+1-n-k}=0 \\
=& \text{coefficient of }  a_1a_2\cdots a_{n-k}(-\beta_1)(-\beta_2)\cdots (-\beta_{n-k}) \text{ in } P\\
 & \text{evaluated on } \alpha_{k+1}=0 \text{ and } b_{m+1-n-k}=0  \/.\\
\end{align*}
When $\alpha_{k+1}=0$ and $b_{m+1-n-k}=0$, we have:
\begin{align*}
P= 
&  (a_1-\alpha_1)(a_1-\alpha_2) \cdots (a_1-\alpha_{k})a_1 \\
& \cdots  \\
&  (a_{n-k}-\alpha_1)(a_{n-k}-\alpha_2) \cdots  (a_{n-k}-\alpha_{k})a_{n-k} \\
&  (b_1-\beta_1)(b_1-\beta_2) \cdots (b_1-\beta_{n-k}) \\
& \cdots  \\
&  (b_{m-n+k}-\beta_1)(b_{m-n+k}-\beta_2) \cdots  (b_{m-n+k}-\beta_{n-k}) \\
&  (-\beta_1)(-\beta_2)\cdots (-\beta_{n-k}) \\
=
&  (a_1-\alpha_1)(a_1-\alpha_2) \cdots (a_1-\alpha_{k}) \\
& \cdots  \\
&  (a_{n-k}-\alpha_1)(a_{n-k}-\alpha_2) \cdots  (a_{n-k}-\alpha_{k}) \\
&  (b_1-\beta_1)(b_1-\beta_2) \cdots (b_1-\beta_{n-k}) \\
& \cdots  \\
&  (b_{m-n+k}-\beta_1)(b_{m-n+k}-\beta_2) \cdots  (b_{m-n+k}-\beta_{n-k}) \\
&   a_1a_2\cdots a_{n-k}(-\beta_1)(-\beta_2)\cdots (-\beta_{n-k}) \\
=& \prod_{i=1}^{n-k}\prod_{j=1}^{k} (a_i-\alpha_j) \prod_{i=1}^{m-n+k}\prod_{j=1}^{n-k} (b_i-\beta_j)   \\
& a_1a_2\cdots a_{n-k}(-\beta_1)(-\beta_2)\cdots (-\beta_{n-k}) \/.
\end{align*}
Hence we have:
\begin{align*}
u^*(B)
=&\prod_{i=1}^{n-k}\prod_{j=1}^{k} (a_i-\alpha_j) \prod_{i=1}^{m-n+k}\prod_{j=1}^{n-k} (b_i-\beta_j)  \\
=& c_{k(n-k)}(S_1^{'\vee}\otimes S'_2)c_{(n-k)(m-n+k)}(Q_1^{'\vee}\otimes Q'_2) \\
=& c_{top}(S_1^{'\vee}\otimes S'_2)c_{top}(Q_1^{'\vee}\otimes Q'_2) \/.
\end{align*}
The same argument shows that
\[
 v^*(A)=c_{top}(S_1^{''\vee}\otimes S^{''}_2)c_{top}(Q_1^{''\vee}\otimes Q^{''}_2),
\]
which concludes the proof.
\end{proof}

Recall that when $k=0$ and $k=n-1$, $\tau_{m,n,0}=\bbP^{mn-1}$ and $\tau_{m,n,n-1}=\bbP^{m-1}\times \bbP^{n-1}$ are smooth. hence $e(m,n,0)=Eu_{\tau_{m,n,0}}(\tau_{m,n,n-1})=1$, and $e(m,n,n-1)=Eu_{\tau_{m,n,n-1}}(\tau_{m,n,n-1})=1$.
Notice that this formula translates directly to the affine cone $\Sigma_{m,n,k}$. Locally we can view $\Sigma_{m,n,k}$ as the product $ \tau_{m,n,k}\times K^*$, and the product property of the local Euler obstruction shows that:
\[
Eu_{\Sigma_{m,n,k}}(\Sigma_{m,n,k+i}^\circ)=Eu_{\tau_{m,n,k}}(\tau_{m,n,k+i}^\circ)\cdot 1=\binom{k+i}{i} \/.
\]
The theorem shows that if we fix $m-n=c$, the local Euler obstruction of $\tau_{m,n,k}$ on its smallest stratum $\tau_{m,n,n-1}$ form the Pascal's triangle of $k$ and $n$, as observed in \cite[Figure 1]{NG-TG}.

\begin{remark}
\label{rmk; pascal}
One can see from the proof that the Pascal's triangle comes from the property of the Chern class $\alpha=c_{top}(S_1^\vee\otimes S_2)c_{top}(Q_1^\vee\otimes Q_2) $ on the Grassmannians $M=G(k+1,n+1)\times G(n-k,m+1)$. The class $\alpha\cap[M]=u^*\alpha'\cap [M_1]+v^*\alpha''\cap [M_2]$ is only contributed by the same constructions on two subvarieties 
$M_1$ and $M_2$, where $M_1=G(k,n)\times G(n-k,m+1)$ and $M_2=G(k+1,n)\times G(n-k-1,m+1)$ are the embeddings of small Grassmannians. 
It should be interesting to study whether the class $\alpha=c_{top}(S_1^\vee\otimes S_2)c_{top}(Q_1^\vee\otimes Q_2)\cap [M]$ bears any geometry or combinatorial property  related to the determinantal varieties $\tau_{m,n,k}$.
\end{remark}

\section{the Characteristic Cycle of the intersection cohomology sheaf of $\tau_{m,n,k}$}
\label{S; char}
When the base field is $\bbC$, we
can use the result of the previous section to compute the characteristic cycle of the intersection
cohomology sheaf of $\tau_{m,n,k}$. In this section we prove that for the determinantal variety $\tau_{m,n,k}$, the characteristic cycle of its intersection cohomology sheaf is irreducible, or equivalently, the microlocal multiplicities are all $0$ except for the top dimensional piece. In this section we will assume our base field to be $\bbC$.

\subsection{Characteristic Cycle of a Constructible Sheaf}
Let $M$ be a smooth  compact complex algebraic variety, and $\sqcup_{i\in I} S_i$ be a Whitney stratification of $M$. For any constructible sheaf $\calF^\bullet$ with respect to the stratification, one can assign a cycle in the cotangent bundle $T^*M$ to $\calF^\bullet$. This cycle is called the \textit{Characteristic Cycle} of $\calF^\bullet$, and is denoted by $\calCC(\calF^\bullet)$. The cycle can be expressed as a Lagrangian cycle
\[
\calCC(\calF^\bullet)=\sum_{i\in I} c_i(\calF^\bullet)[T^*_{\overline{S_i}}M] \/.
\]
Here the integer coefficients $c_i(\calF^\bullet)$ are called the \textit{Microlocal Multiplicities}, and are explicitly constructed in \cite[Section 4.1]{Dimca} using the $k$th Euler obstruction of pairs of strata and the stalk Euler characteristic $\chi_i(\calF^\bullet)$. For any $x\in S_i$, the stalk Euler characteristic $\chi_i(\calF^\bullet)$ are defined as:
\[
\chi_i(\calF^\bullet)=\sum_{p} (-1)^p\dim H^p(\calF^\bullet(x)).
\]

For the stratification  $\sqcup_{j\in I} S_j$ of $M$, we define 
\[
e(j,i)=Eu_{\overline{S_i}}(S_j):=Eu_{\overline{S_i}}(x); x\in S_j
\]
to be the local Euler obstruction along the $j$th stratum $S_j$ in the closure of $S_i$. If $S_j\not\subset \overline{S_i}$, then we define $e(j,i)=0$.

The following deep theorem from \cite[Theorem 3]{Du}, \cite[Theorem 6.3.1]{Ka} reveals the relation the microlocal multiplicities $c_i(\calF^\bullet)$, the stalk Euler characterictic $\chi_i(\calF^\bullet)$, and the local Euler obstructions $e(i,j)$ :
\begin{thm}[Microlocal Index Formula]
\label{thm; microlocal}
For any $j\in I$, we have the following formula:
\[
\chi_j(\calF^\bullet)=\sum_{i\in I}e(j,i)c_i(\calF^\bullet)  \/.
\]
\end{thm}

This theorem suggests that if one knows about any two sets of the indexes, then one can compute the third one. 

\subsection{Intersection Cohomology Sheaf}
For any subvariety $X\subset M$, Goresky and MacPherson defined a sheaf of bounded complexes $\calIC$ on $M$ in \cite{G-M2}. This sheaf is determined by the structure of $X$, and is usually called the \textit{Intersection Cohomology Sheaf} of $X$. For details about the intersection cohomology sheaf and intersection homology one is refereed to \cite{G-M1} and \cite{G-M2}. This sheaf is constructible with respect to any Whitney stratification on $M$, hence one can define its characteristic cycle $\calCC(\calIC)$.
We define $\calCC(\calIC)$ to be the \textit{Characteristic Cycle of the Intersection Cohomology Sheaf of $X$}, and we will call it the $IC$ characteristic cycle for short.

When $X$ admits a small resolution, the intersection cohomology sheaf $\calIC$ is derived from the constant sheaf on the small resolution.
\begin{thm}[Goresky, MacPherson {{\cite[\S 6.2]{G-M2}}}]
\label{thm; stalk}
Let $X$ be a $d$-dimensional irreducible complex algebraic variety. Let $p\colon Y\to X$ be a small resolution of $X$, then 
\[
\calIC\cong Rp_*\bbC_Y[2d] .
\]
In particular, for any point $x\in X$ the stalk Euler characteristic of $\calIC$ equals the Euler characteristic of the fiber:
\[
\chi_x (\calIC)=\chi (p^{-1}(x)).
\]
\end{thm}

\subsection{Characteristic Cycle of the Intersection Cohomology Sheaf of $\tau_{m,n,k}$} 
\label{S; charcycle}
In terms of the determinantal varieties $\tau_{m,n,k}$, Proposition~\ref{prop; small} shows that $\nu\colon \hat\tau_{m,n,k}\to \tau_{m,n,k}$ is a small resolution. Let $X:=\tau_{m,n,k}\subset \bbP^{mn-1}:=M$ be the embedding. Since $\bbP^{mn-1}=\tau_{m,n,0}$, we consider the Whitney stratification $\{S_i:=\tau_{m,n,i}^\circ| i=0,1,\cdots ,n-1\}$ of $\bbP^{mn-1}$. One has the following observations:
\begin{enumerate}
\item $\overline{S_i}=\tau_{m,n,i}$ ; 
\item $X=\tau_{m,n,k}=\overline{S_k}$ ; 
\item $S_i\subset \overline{S_j}$ if and only if $i\geq j$. 
\end{enumerate}

The main result of this section is the following theorem:
\begin{thm}
\label{thm; charcycle}
For $i=0,1,\cdots ,n-1$, let $c_i:=c_i(\calICT)$ be the microlocal multiplicities. Then
\[
c_i = 
\begin{cases} 
0, & \mbox{if } i\neq k \\ 
1, & \mbox{if } i=k
\end{cases} \/.
\] 
Hence we have:
\[
\calCC(\calICT)=[T^*_{\tau_{m,n,k}}\bbP^{mn-1}]
\]
is irreducible.
\end{thm}
\begin{proof}
Let $\nu\colon \hat\tau_{m,n,k}\to \tau_{m,n,k}$ be the small resolution. For any $p\in S_j=\tau_{m,n,j}^\circ$, we have $\nu^{-1}(p)=G(k,j)$.
Theorem~\ref{thm; stalk} shows that for any $j=0,1,\cdots ,n-1$, any $p\in S_j$, the stalk Euler characteristic equals:
\[
\chi_j(\calICT)=\chi_{p}(\calICT)=\chi(\nu^{-1}(p))=\binom{j}{k} \/.
\]
And Theorem~\ref{thm; pascal} shows that:
\[
e(j,i)=Eu_{\overline{S_i}}(S_j)=Eu_{\tau_{m,n,i}}(\tau_{m.n.j}^\circ)=\binom{j}{i} \/.
\]
Here when $a<b$ we make the convention that $\binom{a}{b}=0$.

Hence the Microlocal Index Formula~\ref{thm; microlocal} becomes:
\[
\chi_j(\calICT)=\binom{j}{k}=\sum_{i\in I}e(j,i)c_i(\calICT)=\sum_{i\in I}\binom{j}{i}c_i \/.
\]
One gets the following linear equation:
\[
  \left[ {\begin{array}{c}
   \binom{0}{k} \\
   \binom{1}{k} \\
   \cdots       \\
   \binom{k}{k} \\
   \binom{k+1}{k} \\
   \cdots        \\
   \binom{n-1}{k} 
  \end{array} } 
  \right]
=
\left[ {\begin{array}{ccccccc}
   \binom{0}{0}    &          0   &          &           &         &    &    \\
   \binom{1}{0}    &\binom{1}{0}  &          &           &         &    &    \\
   \cdots          &     \cdots   &          &           &         &    &    \\
   \binom{k}{0}    &  \binom{k}{1}& \cdots   & \binom{k}{k} &         &    &  \\                
   \binom{k+1}{0}  & \binom{k+1}{1}   &   \cdots        &    \binom{k+1}{k} & \binom{k+1}{k+1}        &    &   \\
   \cdots          &      \cdots    &     \cdots      &     \cdots    &    \cdots    &  \cdots  & \\
   \binom{n-1}{0}  & \binom{n-1}{1}   &    \cdots      &   \cdots    & \cdots  &\cdots   & \binom{n-1}{n-1}\\
  \end{array} } 
  \right]   
\cdot
\left[ {\begin{array}{c}
   c_0 \\
   c_1 \\
   \cdots       \\
   c_k\\
   c_{k+1} \\
   \cdots        \\
   c_{n-1} 
  \end{array} } 
  \right]   \/.   
\]
Notice that the middle matrix is invertible, since the diagonals are all non-zero, hence the solution vector $[c_0,c_1,\cdots ,c_{n-1}]^T$ is unique. Hence 
\[
c_i = 
\begin{cases} 
0, & \mbox{if } i\neq k \\ 
1, & \mbox{if } i=k
\end{cases} 
\]  
is the unique solution to the system.
\end{proof}
This theorem shows that the $IC$ characteristic cycle of the determinantal variety $\tau_{m,n,k}$  equals the conormal cycle $[T^*_{\tau_{m,n,k}}\bbP^{mn-1}]$. As pointed out in \cite[\S4.2]{MR2097164}, the Chern-Mather class of $\tau_{m,n,k}$ can be treated as the `shadow' of the projectived conormal cycle $[\bbP(T^*_{\tau_{m,n,k}}\bbP^{mn-1})]$. Hence if we know the Chern-Mather class of $\tau_{m,n,k}$, one can know explicitly what the $IC$ characteristic cycle $\calCC(\calICT)$ is. 

\section{Chern-Mather class of $\tau_{m,n,k}$}
\label{cmather}
\subsection{Chern-Mather class}
Since we know about the Nash Blowup of $\tau_{m,n,k}$ and the Nash tangent bundle $\calT$, one natural task is to compute the Chern-Mather class of $\tau_{m,n,k}$. Notice that $N_{m,n,k}$ is the projective vector bundle $\bbP(Q_1^\vee\otimes S_2)$ over $G(k,n)\times G(n-k,m)$, the computation would consist of integrations of certain Chern classes over a product of Grassmannians $G(k,n)\times G(n-k,m)$. 

However, the explicit expression of the Local Euler obstruction suggests a more efficient way: we can use the Tjurina transform $\hat\tau_{m,n,k}$ introduced in \S\ref{Tjurina}, instead of the Nash Blowup, and the computations will only involve a single Grassmannian.
\begin{lemma}
Let $\nu \colon \hat\tau_{m,n,k}\to \tau_{m,n,k}$ be the Tjurina transform of $\tau_{m,n,k}$. Then the Chern-Mather class of $\tau_{m,n,k}$ equals
\begin{align*}
c_M(\tau_{m,n,k})
&=\nu_*(c(T_{\hat\tau_{m,n,k}})\cap [\hat\tau_{m,n,k}])  \/.
\end{align*}
\end{lemma}
\begin{proof}
First we assume our base field is algebraically closed of characteristic $0$, and let $c_*$ be MacPherson's natural transformation as defined in \S\ref{S; c_*}.
For any $\varphi\in\tau_{m,n,k+i}^\circ$, the fiber of $\hat\tau_{m,n,k}$ at $\varphi$ is $\nu^{-1}(\varphi)\cong G(k,k+i)$. Notice that the local Euler obstruction of $\tau_{m,n,k}$ at $\varphi$ equals
\[
Eu_{\tau_{m,n,k}}(\varphi)=\binom{k+i}{k}=\chi(G(k,k+i)) \/.
\]
Hence we have 
\begin{align*}
\nu_*(\ind_{\hat\tau_{m,n,k}})
&=\sum_{i=0}^{n-1-k} \chi(G(k,k+i)) \ind_{\tau_{m,n,k+i}^\circ}\\
&= \sum_{i=0}^{n-1-k} \binom{k+i}{k} \ind_{\tau_{m,n,k+i}^\circ} \\
&= Eu_{\tau_{m,n,k}} \/.
\end{align*}
Recall that for any variety $X$, $c_M(X)=c_*(Eu_X)$. By the functorial property of $c_*$ one gets
\begin{align*}
c_M(\tau_{m,n,k})
&=\nu_* c_*(\ind_{\hat\tau_{m,n,k}}) \\
&=\nu_*(c(T_{\hat\tau_{m,n,k}})\cap [\hat\tau_{m,n,k}]) \/.
\end{align*}
For an arbitrary algebraically closed field $K$, let $N_{m,n,k}$ be the Nash Blowup of $\tau_{m,n,k}$. Write the Chern-Mather class of $\tau_{m,n,k}$ as
\[
c_M(\tau_{m,n,k})=\sum_{i=0}^{mn-1} \beta_i [\bbP^{i}] \in A_*(\bbP^{mn-1}) \/.
\] 
One has 
\[
\beta_i=\int_{\tau_{m,n,k}} c_1(\calO(1))^i \cap c_M(\tau_{m,n,k}) \/.
\]
Hence it amounts to prove that for any $i=0,1,\cdots , mn-1$, the following identity holds true:
\begin{align*}
I_i
&:= \int_{N_{m,n,k}} c(\calT)c_1(\calO(1))^i\cap [N_{m,n,k}] \\
&- \int_{\hat\tau{m,n,k}} c(T_{\hat\tau{m,n,k}})c_1(\calO(1))^i\cap [\hat\tau{m,n,k}] \\
&=0 \/. 
\end{align*}
Let $q\colon N_{m,n,k}\to \hat\tau_{m,n,k}$ be the projection onto $\hat\tau_{m,n,k}$ away from the $G(n-k,m)$ component. It is both proper and flat, hence by the projection formula we have
\begin{align*}
I_i=\int_{N_{m,n,k}} (c(\calT)- c(q^*T_{\hat\tau{m,n,k}}))c_1(\calO(1))^i\cap [N_{m,n,k}]  \/.
\end{align*}
Since $N_{m,n,k}=\bbP(Q_1^{\vee}\otimes S_2)$ is the projective bundle over $G(k,n)\times G(n-k,m)$, $I_i$ can be written into a polynomial whose variables are the Chern roots of the universal subbundles $S_1,S_2$ and the universal quotient bundles $Q_1,Q_2$. This is an identity in the polynomial ring of integer coefficients, which is independent of the characteristic of the base field $K$. Hence $I_i=0$ holds true over characteristic $0$ fields shows that it holds true for arbitrary algebraically closed base field.
\end{proof}

\begin{remark}
When the base field is $\bbC$, the above lemma can be deduced from the following theorem.
\begin{thm*}[Theorem 3.3.1 \cite{Jones}]
Let $X$ be a subvariety of a smooth space $M$. Let $p\colon Y\to X$ be a small resolution of $X$. Under the assumption that $\calCC(\calIC)$ is irreducible, i.e., 
$\calCC(\calIC)=[\overline{T_{X_{sm}}^*M}]$ is the conormal cycle of $\tau_{m,n,k}$.
Then 
\[
c_M(X)=p_*(c(T_Y)\cap [Y]) \/.
\]
\end{thm*}
Just as we have proved in \S\ref{S; char}, the assumption holds true for $\tau_{m,n,k}$, and the Tjurina transform $\hat\tau_{m,n,k}$ is indeed a small resolution of $\tau_{m,n,k}$.

This theorem shows that under this assumption we can use any small resolution and its tangent bundle to compute the Chern-Mather class. However, as pointed out in \cite[Rmk 3.2.2]{Jones}, this is a rather unusual phenomenon. It is known to be true for Schubert varieties in a Grassmannian, for certain Schubert varieties in flag manifolds of types B, C, and D, and for theta divisors of Jacobians (cf. \cite{MR1642745}). 
By the previous section, determinantal varieties also share this 
property.
\end{remark}

Actually this class $\nu_*(c(T_{\hat\tau_{m,n,k}})\cap [\hat\tau_{m,n,k}])$ has been studied in \cite{XP1}, as the bridge to get the Chern-Schwartz-MacPherson class of $\tau_{m,n,k}$. Now knowing that this is indeed the Chern-Mather class of $\tau_{m,n,k}$, the first part of the main Theorem in \cite{XP1} can be re-stated as:
\begin{thm}
\label{thm; cmalgorithm}
Let $S$ and $Q$ be the universal sub and quotient bundle over the Grassmannian $G(k,n)$. For $k\geq 1$, $i,p=0,1\cdots ,m(n-k)$, we let 
\begin{align*}
A_{i,p}(k)= A_{i,p}(m,n,k) &:=\int_{G(k,n)} c(S^\vee\otimes Q)c_i(Q^{\vee m})c_{p-i}(S^{\vee m}) \cap [G(k,n)] \\
B_{i,p}(k)=B_{i,p}(m,n,k) &:=\binom{m(n-k)-p}{i-p}
\end{align*}
and let
\[
A(k)=A(m,n,k)=\left[ A_{i,p}(k) \right]_{i,p}\quad, \quad B(k)=B(m,n,k)=\left[ B_{i,p}(k) \right]_{i,p}
\]
be $m(n-k)+1 \times m(n-k)+1$ matrices. Here we let $\binom{a}{b}=0$ if $a<b$ or $a<0$ or $b<0$.
Then
\[
c_M(\tau_{m,n,k})=\tr(A(k)\cdot \calH(k)\cdot B(k)) \/.
\]
\end{thm}

\subsection{Projectived Conormal Cycle}
Let $X\subset M$ be a subvariety of $m$-dimensional ambient space. The \textit{projectived conormal cycle} of $X$ is a $m-1$-dimensional cycle in the total space $\bbP(T^*M)$, defined as $Con(X):=[\bbP(T^*_X M)]$. For the determinantal variety $X=\tau_{m,n,k}\subset \bbP^N$, $Con(X)$
is a $N-1$ dimensional cycle in $\bbP(T^* \bbP^N)$. Recall that $\bbP(T^* \bbP^N)$ can be realized as a divisor
of class $h_1+h_2$ in $\bbP^N\times \bbP^N$. Here $N=mn-1$. As we have proved in Theorem~\ref{thm; charcycle}, the (affine) conormal cycle is indeed the $IC$ characteristic cycle of $\tau_{m,n,k}$. 

Write the Chern-Mather class of $\tau_{m,n,k}$ as
\[
c_M(\tau_{m,n,k}) = \sum_{l=0}^N \beta_l h_1^{N-l} \cap [\bbP^N] \/.
\]
Tracing definitions (see \cite[\S4.2]{MR2097164}) shows that
\begin{align*}
Con(\tau_{m,n,k})
&=  (-1)^{(m+k)(n-k)-1}\sum_{l=0}^{N-1} (-1)^l\beta_l (h_1+h_2)^{l+1} h_1^{mn-l}\cap [\bbP^N \times \bbP^N] \\
&=  (-1)^{(m+k)(n-k)-1}\sum_{j=1}^{N-1} \sum_{l=j-1}^{N-1} (-1)^l\beta_l\binom{l+1}{j} h_1^{mn-j}h_2^{j}\cap [\bbP^N \times \bbP^N]\\
&=  (-1)^{(m+k)(n-k)-1}\sum_{j=1}^{mn-2} \sum_{l=j-1}^{mn-2} (-1)^l\beta_l\binom{l+1}{j} h_1^{mn-j}h_2^{j}\cap [\bbP^N \times \bbP^N] \/.
\end{align*}
Hence the Chern-Mather class determines the conormal cycle as a class in $\bbP^N \times \bbP^N$.

For projective varieties, as pointed in \cite[Section 2.4]{Aluffi}, the coefficients of the conormal cycle correspond to the degree of the polar classes. 
Let $X\subset \bbP^N$ be a $d$-dimensional projective variety, the $k$-th polar class of $X$ is defined as follows:
\[
[M_k]:=\overline{\{x\in X_{sm} | \dim (T_x X_{sm}\cap L_k)\geq k-1  \}} \/.
\]
Here $L_k\subset L^N$ is a linear subspace of codimension $d-k+2$. This rational equivalence class is independent of $L_k$ for general $L_k$. The polar classes are closely connected to the Chern-Mather class of $X$, as pointed out in \cite{Piene}:
\begin{thm}[Theorem 3~\cite{Piene}]
The $k$-th polar class of $X$ is given by
\[
[M_k]=\sum_{i=0}^k (-1)^i \binom{d-i+1}{d-k+1}  H^{k-i}\cap c^i_M(X).
\]
Here $c_M^i(X)$ is the codimension $i$ piece of the Chern-Mather class of $X$.
\end{thm}
In terms of the determinantal varieties $\tau_{m,n,k}$, if we write 
\[
c_M(\tau_{m,n,k}) = \sum_{l=0}^N \beta_l [\bbP^l],
\]
then the $l$-th polar class of $\tau_{m,n,k}$ is
\begin{align*}
[M_l]
=& \sum_{i=0}^{l} (-1)^i \binom{(m+k)(n-k)-i}{(m+k)(n-k)-l} \beta_{(m+k)(n-k)-1-i} 
[\bbP^{(m+k)(n-k)-1-l}] \/.
\end{align*}

\subsection{Examples}
\label{example}
\subsubsection{Chern-Mather class}
Using the Schubert2 package in Macaulay2~\cite{M2}, we are able to compute the Chern-Mather class of determinantal varieties. Here are some examples.
\begin{enumerate}
\item $\tau_{3,3,s}$ \\
\begin{tabular}{c | ccccccccc}
  Table  & $\bbP^{0}$ & $\bbP^{1}$  & $\bbP^{2}$  & $\bbP^{3}$  & $\bbP^{4} $  & $\bbP^{5} $ & $\bbP^{6}$  & $\bbP^{7}$ & $\bbP^{8}$ \\
\hline
$c_M(\tau_{3,3,0})$ & 9  & 36 & 84  & 126 & 126  & 84 & 36 & 9 & 1 \\
$c_M(\tau_{3,3,1})$ & 18 & 54 & 102 & 126 & 102  & 54 & 18 & 3 & 0 \\
$c_M(\tau_{3,3,2})$ & 9  & 18 & 24  & 18  & 6    &  0 &  0 & 0 & 0\\
\end{tabular}

\item $\tau_{4,3,s}$. \\
\begin{small}
\begin{tabular}{c | *{12}{c}}
  Table  & $\bbP^0$ & $\bbP^1$ & $\bbP^2$ & $\bbP^3$ & $\bbP^4$ & $\bbP^5$ & $\bbP^6$ & $\bbP^7$ &$\bbP^8$ & $\bbP^9$ & $\bbP^{10}$ & $\bbP^{11}$  \\
\hline
$c_M(\tau_{4,3,0})$       & 12  & 66 & 220 & 495 & 792 & 924  & 792 & 495 & 220& 66 & 12& 1 \\
$c_M(\tau_{4,3,1})$       & 24  & 96 & 248 & 444 & 564 & 514  & 336 & 153 & 44 & 6  & 0 & 0  \\
$c_M(\tau_{4,3,2})$       & 12  & 30 & 52  & 57  & 36  & 10   &  0  & 0   & 0  & 0  & 0 & 0  \\
\end{tabular}
\end{small}
\item $\tau_{4,4,s}$ \\
\begin{tabular}{c | *{8}{c}}
  Table  & $\bbP^{0}$ & $\bbP^{1}$  & $\bbP^{2}$  & $\bbP^{3}$  & $\bbP^{4} $  & $\bbP^{5} $ & $\bbP^{6}$  & $\bbP^{7}$  \\
\hline
$c_M(\tau_{4,4,1})$ &48 &288 &1128 &3168 &6672 & 10816 & 13716 & 13716    \\
$c_M(\tau_{4,4,2})$ &48 &216 &672 &1524 &2592 &3368 &3376  &2602     \\
$c_M(\tau_{4,4,3})$ &16 &48  &104 &152  &144  &80   & 20   & 0       \\
\end{tabular}
\vskip .05in
\begin{tabular}{c | *{8}{c}}
  Table  & $\bbP^{8}$  & $\bbP^{9}$ & $\bbP^{10}$ & $\bbP^{11}$ & $\bbP^{12}$ & $\bbP^{13}$ & $\bbP^{14}$ & $\bbP^{15}$   \\
\hline
$c_M(\tau_{4,4,1})$ &10816 &6672 &3168 &1128 &288 &48  & 4 & 0 \\
$c_M(\tau_{4,4,2})$ &1504 &616   &160  & 20  & 0  & 0  & 0 & 0  \\
$c_M(\tau_{4,4,3})$ & 0   & 0   & 0   & 0   & 0  & 0  & 0 & 0  \\
\end{tabular}
\end{enumerate}

\subsubsection{Projectived Conormal Cycle}
\begin{enumerate}
\item $\tau_{4,4,s}$ \\
\begin{small}
\begin{tabular}{c | *{7}{c}}
 Table & $h_1^{15}h_2$ & $h_1^{14}h_2^2$ & $h_1^{13}h_2^3$ & $h_1^{12}h_2^4$ & $h_1^{11}h_2^5$ & $ h_1^{10}h_2^6 $  & $h_1^9h_2^7$ \\
\hline
$Con(\tau_{4,4,1})$     & 0  & 0 & 0& 0& 0 & 0 & 0  \\
$Con(\tau_{4,4,2})$     & 0  & 0 & 0 & 20 & 80 & 176 & 256 \\
$Con(\tau_{4,4,3})$     & 4  & 12 & 36 & 68 &  84 &  60 &  20    \\
\end{tabular}
\vskip .05in
\begin{tabular}{c | *{8}{c}}
 Table &$ h_1^8h_2^8$ & $h_1^7h_2^9$ & $h_1^6h_2^{10}$ & $h_1^5h_2^{11}$ & $h_1^4h_2^{12}$ & $h_1^3h_2^{13} $ & $h_1^2h_2^{14}$ & $h_1h_2^{15}$ \\
\hline
$Con(\tau_{4,4,1})$    & 0   & 20 & 60 & 84 & 68 & 36 & 12 & 4     \\
$Con(\tau_{4,4,2})$    & 286   & 256 & 176 & 80  & 20 & 0  & 0  & 0 \\
$Con(\tau_{4,4,3})$    & 0     &   0 &  0  &  0  &  0 &  0 &  0 & 0    \\
\end{tabular}
\end{small}
\item $\tau_{5,4,s}$ \\ 
\begin{tiny}
\hspace*{-0.5cm}
\begin{tabular}{c | *{9}{c}}
  Table  & $h_1^{19}h_2$ & $h_1^{18}h_2^2$ & $h_1^{17}h_2^3$ & $h_1^{16}h_2^4$ & $h_1^{15}h_2^5$ & $ h_1^{14}h_2^6 $  & $h_1^{13}h_2^7$&$ h_1^{12}h_2^8$ & $h_1^{11}h_2^9$ \\
\hline
$Con(\tau_{5,4,1})$       & 0  & 0  & 0   & 0  & 0  & 0   & 0    & 0    &  0    \\
                         
$Con(\tau_{5,4,2})$       & 0  & 0  & 0   & 0  & 0  & 50 & 240 & 595 & 960  \\
                         
$Con(\tau_{5,4,3})$       & 0  & 10 & 40  & 105 & 176 & 190 & 120 & 35   & 0      \\
\end{tabular}
\vskip .05in 
\begin{tabular}{c | *{10}{c}}
\hspace*{-1cm}
  Table   & $h_1^{10}h_2^{10}$ & $h_1^9h_2^{11}$ & $h_1^8h_2^{12}$ & $h_1^7h_2^{13} $ & $h_1^6h_2^{14}$ & $h_1^5h_2^{15}$ & $h_1^4h_2^{16}$ & $h_1^3h_2^{17}$ & $h_1^2h_2^{18}$ & $h_1h_2^{19}$ \\
\hline
$Con(\tau_{5,4,1})$   &   0      & 0    & 35  & 120 & 190 & 176 & 105 & 40 & 10 & 0    \\
                         
$Con(\tau_{5,4,2})$   & 1116    & 960 & 595 & 240 & 50  & 0    & 0    & 0   & 0   & 0  \\
                         
$Con(\tau_{5,4,3})$   & 0        & 0    & 0    & 0    & 0    & 0    & 0    & 0   & 0   & 0  
\end{tabular}
\end{tiny}
\end{enumerate}

\newpage

\bibliography{eu}

\end{document}